\documentclass{amsart}[12]

\usepackage{geometry}
\usepackage{amsmath}
\usepackage{amsthm}
\usepackage{mathrsfs}
\usepackage{amsfonts}
\usepackage{amssymb}
\usepackage{stmaryrd}
\usepackage{amscd}
\usepackage{array}
\usepackage{amssymb}
\usepackage[all, cmtip]{xy}
\usepackage[dvipsnames]{xcolor}
\usepackage{tikz}
\usetikzlibrary{arrows}
\usetikzlibrary{cd}
\usepackage{enumitem}
\usepackage{blkarray}
\usepackage{hyperref}
\usepackage{mathtools}
\usepackage{tabu}

\newtheorem{theorem}{Theorem}[subsection]
\newtheorem{lemma}[theorem]{Lemma}
\newtheorem{conjecture}[theorem]{Conjecture}
\newtheorem{definition}[theorem]{Definition}

\newtheorem{proposition}[theorem]{Proposition}
\theoremstyle{remark}
\newtheorem{remark}[theorem]{Remark}
\newtheorem{example}[theorem]{Example}
\newtheorem{notation}[theorem]{Notation}

\newcommand{\QQ}{\mathbb{Q}}

\newcommand{\PP}{\mathbb{P}}
\newcommand{\ZZ}{\mathbb{Z}}

\renewcommand{\AA}{\mathbb{A}}

\newcommand{\OO}{\mathscr{O}}

\newcommand{\wP}[1]{\mathbb{P}^1_{a,1}}

\renewcommand{\sslash}{\mathord{/\mkern-6mu/}}

\usepackage{etoolbox}
\let\tinymatrix\smallmatrix

\patchcmd{\tinymatrix}{\scriptstyle}{\scriptscriptstyle}{}{}
\patchcmd{\tinymatrix}{\scriptstyle}{\scriptscriptstyle}{}{}
\patchcmd{\tinymatrix}{\vcenter}{\vtop}{}{}
\patchcmd{\tinymatrix}{\bgroup}{\bgroup\scriptsize}{}{}
\newcommand{\bmu}{\pmb{\mu}}
\newcommand{\bi}{\mathfrak{i}}

\newcommand{\Gm}{\mathbb{G}_m}

\newcommand{\age}{\mathrm{age}}
\newcommand{\cX}{\mathcal{X}}

\newcommand{\cC}{\mathcal{C}}
\newcommand{\cL}{\mathcal{L}}
\renewcommand{\L}{\mathscr{L}}

\newcommand{\In}[1]{\mathcal{I}(#1)}

\usepackage[backend=biber,
style=alphabetic,
isbn=false,
url=false,
sorting=nyt]{biblatex}
\usepackage{bm}
\newcommand{\CR}{\mathrm{CR}}

\renewcommand{\AA}{\mathbb{A}}

\newcommand{\vir}{\mathrm{vir}}

\newcommand{\one}{\mathbf{1}}

\newcommand{\tw}{\mathrm{tw}}

\usepackage{bbm}
\usepackage{dsfont}

\DeclareMathOperator{\Pic}{Pic}

\DeclareMathOperator{\Hom}{Hom}

\usepackage[backend=biber,
style=alphabetic,
isbn=false,
url=false,
sorting=nyt]{biblatex}
\usepackage{bm}
\usepackage{xcolor}

\newcommand{\NOTEoff}{\newcommand{\Commentn}[1]{}}

\newcommand\TODOon{\newcommand{\Comment}[1]{\noindent\color{red}{\texttt TODO: }##1\color{black}}}

\newcommand\RACHELon{\newcommand{\Commentr}[1]{\noindent\color{blue}{\texttt Rachel: }##1\color{black}}}

\newcommand\Nawazon{\newcommand{\Commentp}[1]{\noindent\color{green}{\texttt Nawaz: }##1\color{black}}}

\begin{document}

\TODOon

\NOTEoff

\Nawazon

\RACHELon

\title{Subtleties of Quantum Lefschetz without convexity}

\author{Nawaz Sultani}
\address[N.Sultani]{Department of Mathematics, East Hall\\
University of Michigan\\
Ann Arbor, CA 48109\\
U.S.A.}
\email{sultani@umich.edu}

\author{Rachel Webb}
\address[R. Webb]{Department of Mathematics\\
University of California, Berkeley\\
Berkeley, CA 94720-3840\\
U.S.A.}
\email{rwebb@berkeley.edu}

\date{\today}

\begin{abstract}
Let $Y\sslash_\theta G$ be a complete intersection in a GIT quotient $X\sslash_\theta G$ cut out by a $G$-representation $E$.
We show that the $E$-twisted quasimap $I$-function recovers invariants of $(Y, G, \theta)$ if its nonequivariant limit exists \textit{before} restriction to $Y\sslash_\theta G$. This corrects a conjecture credited to Coates-Corti-Iritani-Tseng. We explain how to correct computations in the literature based on the faulty conjecture.
\end{abstract}


\maketitle

\section{Introduction}

Let $X=\AA^n$, let $G$ be a connected complex reductive group acting on $X$, let $\theta$ be a character of $G$, and define 
\[X\sslash_\theta G := [X^{ss}_\theta(G) / G]\]
to be the stacky GIT quotient, where $X^{ss}_\theta(G)$ is the set of semistable points (in the sense of \cite{king}). Assume that all semistable points are stable and that this locus is nonempty. Let $E$ be a $G$-representation, let $s$ be a regular section of the bundle $E_[X/G] \to [X/G]$ induced by $E$, and let $Y$ be the zero locus of $s$ with embedding
$\bi: Y\sslash_\theta G \to X\sslash_\theta G$.
Assume that the tuple $(X, G, \theta, E, s)$ is a \textit{CI-GIT presentation} in the sense of \cite{SW22a}, so that the quasimap $I$-function \eqref{eq:Ifunc-formula} is on the Lagrangian cone of $Y\sslash_\theta G$.

On the other hand, the torus $\Gm$ acts on the total space of $E$ by scaling the fibers, and there is an associated \textit{equivariant twisted quasimap $I$-function} $I^{X\sslash G, E}$ (see \eqref{eq:twisted} below). Let $\kappa$ denote the equivariant parameter of the $\Gm$ action. It follows from \cite[Thm~22]{CCIT19} and \cite[Thm~1.1]{coates-short} that if $G$ is abelian, $E$ is convex, and the nonequivariant limit of $\lim_{\kappa \to 0}I^{X\sslash G, E}$ exists, then $\lim_{\kappa \to 0}I^{X\sslash G, E}$ lies on the Lagrangian cone of $Y\sslash_\theta G$.

In the conclusion of \cite{failure}, the authors observe that if $E$ satisfies ``certain mild conditions'' weaker than convexity, the limit $\lim_{\kappa \to 0} I^{X\sslash G, E}$ still exists, and they hope for a relationship with invariants of ${Y\sslash_\theta G}$ in this case. Their hope is made precise in \cite{OP} as follows, where it is credited to Coates-Corti-Iritani-Tseng.

\begin{conjecture}[Conjecture 5.2 of \cite{OP}]\label{conj:ql}
Assume $G$ is abelian and $Y\sslash G$ has projective coarse moduli space. If $\lim_{\kappa \to 0} \bi^* I^{X\sslash G, E}$ exists, then $\lim_{\kappa \to 0} \bi^* I^{X\sslash G, E}$ is on the Lagrangian cone of $Y\sslash G$.
\end{conjecture}

Recently, multiple authors have used quasimap theory to investigate Gromov-Witten invariants of $Y\sslash_\theta G$ without convexity hypotheses (see \cite{Wang, Webb21, JSZ}). Their theorems allow us to disprove Conjecture \ref{conj:ql} by means of an explicit example (Section \ref{sec:failure}). On the other hand, we can prove that a more modest version of Conjecture \ref{conj:ql} is true (even for nonabelian groups $G$ and nonproper $Y\sslash G$):

\begin{theorem}\label{thm:main}
If $\lim_{\kappa \to 0} I^{X\sslash G, E}$ exists, then $ \bi^* \lim_{\kappa \to 0}I^{X\sslash G, E}$ is on the image of the Lagrangian cone of $Y\sslash G$.
\end{theorem}

In Conjecture \ref{conj:ql} and Theorem \ref{thm:main} we are suppressing a specialization of Novikov variables (see Remark \ref{rmk:nov-rings}).

Conjecture \ref{conj:ql} was used in \cite{OP} to compute quantum periods of orbifold del Pezzo surfaces. In Section \ref{sec:criterion} we give a positivity criterion that can be checked to verify the computations in \cite{OP}, and we check that this criterion holds for the example in \cite[Sec~5.4]{OP}.

\subsection{Acknowledgements} The second author was partially supported by an NSF Postdoctoral Research Fellowship, award number 200213.

\section{Definitions of $I$-functions}

We refer the reader to \cite[Sec~2]{SW22a} for a more detailed exposition of this section. Let $\PP^1_{a, 1}$ be the orbifold line with one stacky point of order $a$.
A \emph{quasimap} to $(X, G, \theta)$ is a representable morphism $q: \wP{a} \to [X/G]$ such that the substack $q^{-1}([X^{us}/G])$ is disjoint from the stacky point in $\wP{a}$, where $X^{us} := X \setminus X^{ss}_\theta(G)$. The \textit{class} of $q$ is the element $\beta \in \Hom(\chi(G), \QQ)$ given by $\beta(\xi) = \deg q^*\cL_\xi$, where $\chi(G)$ is the character group of $G$ and $\cL_\xi$ is the line bundle induced by $\xi$. An element $\beta \in \Hom(\chi(G), \QQ)$ is called \textit{$I$-effective} if it is the class of a quasimap.
 
 Let $T\subset G$ be a maximal torus. Define
 \[
 Y\sslash_G T := [Y^{ss}_\theta(G)/T] \quad \quad \quad  X\sslash_G T := [X^{ss}_\theta(G)/T]
 \]
and let $\varphi: Y\sslash_G T \to Y\sslash G$ be the natural morphism.
There is a group homomorphism
$\Hom(\chi(T), \QQ) \to T$ sending
$\tilde \beta $ to $g_{\tilde \beta}$,
where for $\xi \in \chi(T)$ we have
\begin{equation}\label{eq:def-g}
\xi(g_{\tilde \beta}) = e^{2\pi i \tilde \beta(\xi)}.
\end{equation}

Let $\xi_1, \ldots, \xi_n$ (resp. $\epsilon_1, \ldots, \epsilon_r$, resp. $\rho_1, \ldots, \rho_m$) denote the weights of $X$ (resp. $E$, resp. the Lie algebra of $G$) with respect to $T$. For $\tilde \beta \in \Hom(\chi(T), \QQ)$ we call $\tilde \beta$ \textit{$I$-nonnegative} if the set 
\[
\{ \;\epsilon \in \{\epsilon_j\}_{j=1}^r \;\mid \;\tilde \beta(\epsilon) \in \ZZ_{<0} \;\}
\]
is empty. If $[W/G]$ is any global quotient stack and $g \in G$, we use $\In{[W/G]}_{g}$ to denote the sector of the inertia stack of $[W/G]$ indexed by the conjugacy class of $g$. We define 
\begin{equation}\label{def:x-tilde-beta-intro}
\begin{gathered}
X^{\tilde \beta}:=\{(x_1, \ldots, x_n) \in X \mid x_i = 0 \;\text{for all}\; i \;\text{such that}\; \tilde \beta(\xi_i) \not \in \ZZ_{\geq 0}\}\\
F_{\tilde \beta}(X\sslash T) := X^{\tilde \beta}\sslash_G T \quad \quad \quad F_{\tilde \beta}(Y\sslash T) := (Y \cap X^{\tilde \beta}) \sslash_G T
\end{gathered}
\end{equation} 
The subspace $X^{\tilde \beta} \subset X$ is contained in the fixed locus of $g_{\tilde \beta}^{-1}$, and we view the stacks $F_{\tilde \beta}(X\sslash T)$ and $F_{\tilde \beta}(Y\sslash T)$ as substacks of the $g_{\tilde \beta}^{-1}$-sectors of the respective inertia stacks.

For $\xi \in \chi(T)$ and $W$ a variety with $T$-action, let $\L_\xi$ be the line bundle on $[W/T]$ with total space $[(W \times \AA^1)/T]$ where $T$ acts as given on $W$ and via the character $\xi$ on $\AA^1$. For $\tilde \beta \in \Hom(\chi(T), \QQ)$ we define operational Chow classes on $Y\sslash_G T$ by
\begin{align*}
C^\circ(\tilde\beta, \xi) 
&:= \begin{cases}
\prod_{\tilde \beta(\xi) < k < 0,  \;k-\tilde\beta(\xi) \in \ZZ}^{}(c_1(\L_{\xi}) + kz) & \tilde\beta(\xi) \leq 0\\
\left[\prod_{0 < k \leq \tilde\beta(\xi),  \;k-\tilde\beta(\xi) \in \ZZ}^{}(c_1(\L_{\xi}) + kz)\right]^{-1} & \tilde\beta(\xi) >0.
\end{cases}\\
C(\tilde\beta, \xi) 
&:= \begin{cases}
c_1(\L_{\xi})C^\circ(\tilde\beta, \xi) & \tilde\beta(\xi) \in \ZZ_{< 0}\\
C^\circ(\tilde\beta, \xi) & else.
\end{cases}\\
\end{align*}

\begin{theorem}[ \cite{Webb21}]\label{thm:Ifunc-formula}
The quasimap $I$-function of the CI-GIT presentation $(X, G, \theta, E, s)$ is a series $I(q, z) = \sum_{\delta} q^\delta I_\delta(z)$, where $\delta \in \Hom(\Pic^G(Y), \QQ)$, and for $\beta \in \Hom(\Pic^G(X), \QQ)$ we have
\begin{equation}\label{eq:Ifunc-formula}
\sum_{\delta \mapsto \beta} \varphi^* I_\delta(z) = \sum_{\tilde \beta \mapsto \beta} I_{\tilde \beta}(z) \quad \text{where} \quad I_{\tilde \beta}(z) \in H^\bullet(\In{Y\sslash_G T}_{g_{\tilde \beta}^{-1}}; \QQ)\otimes \QQ[z, z^{-1}]
\end{equation}
where the second sum is over $\tilde \beta \in \Hom(\chi(T), \QQ)$ mapping to $\beta$. Moreover,
\begin{itemize}
\item If $\tilde \beta \in \Hom(\chi(T), \QQ)$ is $I$-nonnegative, we have
\begin{equation}\label{eq:Ifunc-coeff-convex}
I_{\tilde \beta}(z) = \left( \prod_{i=1}^m C(\tilde \beta, \rho_i)^{-1}\right)
\left(\prod_{j=1}^r C(\tilde \beta, \epsilon_j)^{-1}\right)
\left(\prod_{\ell=1}^{n}C(\tilde \beta, \xi_\ell)\right)
\one_{g_{\tilde \beta}^{-1}} 
\end{equation}
where $\one_{g_{\tilde \beta}^{-1}}$ is the fundamental class of the sector $\In{Y\sslash_G T}_{g_{\tilde \beta}^{-1}}$.
\item If the inclusion $F^0_{\tilde \beta}(Y\sslash T) \hookrightarrow F^0_{\tilde \beta}(X\sslash T)$ is l.c.i. of codimension 
\[\# \{\epsilon_j \mid \tilde \beta(\epsilon_j) \in \ZZ_{\geq 0} \},\] 
then \eqref{eq:Ifunc-coeff-convex} holds after replacing $C(\tilde \beta, \epsilon_j)^{-1}$ by $C^\circ(\tilde \beta, \epsilon_j)^{-1}$, replacing $C(\tilde \beta, \xi_j)^{-1}$ by $C^\circ(\tilde \beta, \xi_j)^{-1}$, and replacing $\one_{g_{\tilde \beta}^{-1}}$ by $[F^0_{\tilde \beta}(Y\sslash T)]$.
\end{itemize}
\end{theorem}

\begin{remark}\label{rmk:nov-rings}
Let $(X, G, \theta, E, s)$ be a CI-GIT presentation and let $I(q, z)$ be the series in Theorem \ref{thm:Ifunc-formula}. Let $r: \Hom(\Pic^G(Y), \QQ) \to \Hom(\Pic^G(X), \QQ)$ be the natural morphism and set $I(r(q), z) = \sum_{\delta} q^{r(\delta)} I_\delta(z)$. Consider the restriction $\mathcal L$ of the Lagrangian cone of $Y\sslash_\theta G$ to the Novikov ring of $\Lambda_{(X, G, \theta)}$ (defined as in \cite[Sec~2.2.2]{SW22a}). By \cite{yang}, the series $I(r(q), z)$ is a point of $\mathcal{L}$.
\end{remark}

On the other hand, let $C(\tilde \beta, \epsilon + \kappa)$ denote the class $C(\tilde \beta, \epsilon)$ but with $c_1(\cL_\epsilon)$ replaced by $c_1(\cL_\epsilon)+\kappa$. 
We define the \textit{equivariant twisted $I$-function} to be the series $I^{X\sslash G, E}(q, z) := \sum_\beta q^\beta I^{X\sslash G, E}_\beta(z)$, where $\beta \in \Hom(\chi(G), \QQ)$ and we set
\begin{equation}\label{eq:twisted}
\varphi^* I^{X\sslash G, E}_\beta(z) := \sum_{\tilde \beta \to \beta} 
\left( \prod_{i=1}^m C(\tilde \beta, \rho_i)^{-1}\right)
\left(\prod_{j=1}^r C(\tilde \beta, \epsilon_j+\kappa)^{-1}\right)
\left(\prod_{\ell=1}^{n}C(\tilde \beta, \xi_\ell)\right)
\one_{g_{\tilde \beta}^{-1}} 
\end{equation}
Note that when $G$ is abelian, this agrees with \cite[(22)]{OP} with $\tau=0$. It is shown in \cite[Sec~4]{CCIT19} that when $G$ is abelian the series $I^{X\sslash G, E}$ lies on the $\Gm$-equivariant twisted Lagrangian cone of $X\sslash G$.

\begin{remark}\label{rmk:homogeneity}
By \cite[Thm~3.9(2)]{CCK15}, the $I$ function of $(X, G, \theta, E, s)$ is homogeneous of degree 0 when we set 
\begin{itemize}
\item $\deg(z) = 1$,
\item $\deg(q^\beta) = \beta(\det(X)\det(\mathfrak{g})\det(E))$ where $\det(\cdot)$ is the character of the determinant of the representation $\cdot$ and $\mathfrak{g}$ is the adjoint representation of $G$, and 
\item $\deg(\alpha) = k + \age_{Y\sslash G}(g)$ for $\alpha \in H^k(\In{Y\sslash G}; \QQ)$ (the age $\age_{Y\sslash G}(g)$ is the rational number defined as in \cite[Sec~2.5.1]{CCK15}). 
\end{itemize}
A direct computation shows that \eqref{eq:twisted} is also homogeneous of degree 0 when in addition we set $\deg(\kappa)=1$.
\end{remark}

\section{Definitions of convexity}
Let $E_{X\sslash G}$ denote the vector bundle on $X\sslash_\theta G$ induced by $E$.
Recall that $E_{X\sslash G}$ is \textit{convex} if for every genus $0$ twisted stable map $f: \cC \to \cX$ (with any number of marked points) we have that $H^1(\cC, f^*E_{X\sslash G})=0.$  
By contrast, we make the following definition which depends on the GIT data $(X, G, \theta)$ and not just on the quotient $X\sslash_\theta G$.

\begin{definition}\label{def:q-convex}
Let $(X, G, \theta)$ be a GIT presentation and let $E$ be a $G$-representation. Let $T \subset G$ be a maximal torus. We say that $E$ is \emph{$I$-convex} if every class $\tilde\beta \in \Hom(\chi(T), \QQ)$ mapping to an $I$-effective class in $\Hom(\chi(G), \QQ)$ is $I$-nonnegative.
\end{definition}
\begin{remark}\label{rmk:well-defined}
Definition \ref{def:q-convex} does not depend on the choice of maximal torus. This is because all maximal tori in $G$ are conjugate and conjugation also acts compatibly on $I$-effective classes and weights.
\end{remark}

\begin{proposition}\label{prop:q-convex}
Let $(X, G, \theta)$ be a GIT presentation and let $E$ be a $G$-representation. The following are equivalent:
\begin{enumerate}
\item $E$ is $I$-convex.
\item $E$ is nonnegative; i.e., for some maximal torus $T \subset G$ we have that $\tilde \beta(\epsilon_i) \geq 0$ for every $\tilde \beta$ mapping to an $I$-effective class and every weight $\epsilon_i$ of $E$.
\item For every quasimap $q:\PP^1_{a, 1} \to [X/G]$, we have $H^1(\PP^1_{a, 1}, q^*E_{[X/G]}) = 0.$ 
\item The nonequivariant limit $\lim_{\kappa \to 0} I^{X\sslash G, E}$ exists.
\end{enumerate}
\end{proposition}

\begin{remark}
In Proposition \ref{prop:q-convex} we could replace ``some maximal torus'' with ``every maximal torus'' in condition (2), for the same reason as in Remark \ref{rmk:well-defined}.
\end{remark}

\begin{proof}[Proof of Proposition \ref{prop:q-convex}]
We first show equivalence of (1) and (4). By definition, $E$ fails to be $I$-convex if and only if some $C(\tilde \beta, \epsilon_j + \kappa)^{-1}$ contains a factor of the form $c_1(\cL_{\epsilon_j})+\kappa$ in its denominator. (The remaining factors of the denominator all have the form $c_1(\cL_{\epsilon_j})+\kappa+ kz$ for some $k \neq 0$.) Expanding $C(\tilde \beta, \epsilon_j + \kappa)^{-1}$ as a Laurent series in $\kappa$, we see that $\lim_{\kappa \to 0} C(\tilde \beta, \epsilon_j+\kappa)^{-1}$ does not exist precisely when we have a denominator of the form $c_1(\cL_{\epsilon_j})+\kappa$.

Assume (1). Suppose for contradiction that we have $\tilde \beta(\epsilon_i) = c < 0$ for some $\tilde \beta$ mapping to an $I$-effective class and for some $\epsilon_i$. This means we have a representable morphism $q: \PP^1_{a, 1} \to [X/T]$ of class $\tilde \beta$ with finite base locus, and we have $a\tilde \beta \in \Hom(\chi(T), \ZZ)$ by \cite[Lem~4.6]{CCK}. By precomposing with the finite representable cover
\[
\PP^1_{1, 1} \to \PP^1_{a, 1} \quad \quad \quad [u:v] \mapsto [u^a: v] \quad t \mapsto t
\]
we get a representable morphism $q': \PP^1_{1, 1} \to [X/G]$ of class $a\tilde \beta$, so $a\tilde \beta$ is $I$-effective. But $a\tilde \beta(\epsilon_i) = ac \in \ZZ_{<0}$, a contradiction. Hence we must have $c \geq 0$ and (2) holds

Assume (2). Let $q: \PP^1_{a, 1} \to [X/G]$ be a quasimap. Then $q$ factors through a quasimap $\tilde q: \PP^1_{a, 1} \to [X/T]$ (see \cite[Rmk~4.1.1]{Webb21}).
Let $\tilde \beta$ denote the class of $\tilde q$ and let $p: [X/T] \to [X/G]$ be the natural quotient map. Then $p^*E_{[X/G]}$ splits as $\oplus_i \L_{\epsilon_i}$ and $\tilde q^*\L_{\epsilon_i}$ has degree $\tilde \beta(\epsilon_i)\geq 0$. By orbifold Riemann-Roch \cite[Thm~7.2.1]{AGV08} we have $H^1(\PP^1_{a, 1}, \tilde q^*\L_{\epsilon_i})$ = 0 . Hence (3) holds.

Assume (3). Suppose $\tilde \beta \in \Hom(\chi(T), \QQ)$ maps to an $I$-effective class. By \cite[Lem~4.3.2]{SW22a} there is a morphism $\tilde q: \PP^1_{a, 1} \to [X/T]$ of class $\tilde \beta$ such that $q:= p \circ \tilde q$ is a quasimap to $[X/G]$ (where $p$ is as defined in the previous paragraph). If $p^*E_{[X/G]} = \oplus_i \L_{\epsilon_i}$, then we know $H^1(\PP^1_{a, 1}, \tilde q^*\L_{\epsilon_i} )=0$ for each $i$, so orbifold Riemann-Roch tells us $\tilde \beta(\epsilon_i) = \deg \tilde q^*\L_{\epsilon_i} \geq -1$. Consider the degree-2 cover $f: \PP^1_{a, 1} \to \PP^1_{a, 1}$ given by squaring each of the homogeneous coordinates. The composition $q \circ f$ is a quasimap to $[X/G]$ such that the associated lift has class $2 \tilde \beta$. If $\tilde \beta(\epsilon_i) = -1$ then $2\tilde \beta(\epsilon_i) = -2 < -1$ and we have a contradiction. Hence $\tilde \beta(\xi_i) > -1$ and (1) holds.

\end{proof}

\begin{remark}
If a class $\tilde \beta$ is $I$-nonnegative and $q$ is a quasimap of class $\tilde \beta$ it is still possible that $H^1(\PP^1_{a, 1}, q^*E_{[X/G]}) \neq 0$. Compare this with Proposition \ref{prop:q-convex} part (3).
\end{remark}

Theorem \ref{thm:main} follows immediately from the proposition.
\begin{proof}[Proof of Theorem \ref{thm:main}]
If the nonequivariant limit $\lim_{\kappa \to 0} I^{X\sslash G, E}$ exists, then $E$ is $I$-convex by Proposition \ref{prop:q-convex}, and by Theorem \ref{thm:Ifunc-formula} we have that $\bi^*\lim_{\kappa \to 0} I^{X\sslash G, E}$ is equal to the quasimap $I$-function of $(X, G, \theta, E, s)$. This $I$-function is on the Lagrangian cone of $Y\sslash_\theta G$ up to a restriction of Novikov rings (Remark \ref{rmk:nov-rings}).
\end{proof}

\begin{notation}
Let $X = \AA^n$ be a representation of $G = (\Gm)^k$. We will identify the weights of the representation with elements of $\ZZ^k$. If $k=1$ we will write $\OO_{X\sslash G}(d)$ for the line bundle $\L_{\sigma_d}$ where $\sigma_d$ is the character of $\Gm$ of weight $d$.
\end{notation}

\begin{example}[$I$-convexity does not imply convexity]
Let $G=\Gm$ act on $X=\AA^3$ with weights $(1,1,3)$, and let $\theta$ be the identity character. Then $X\sslash_\theta G$ is equal to the weighted projective space $\PP(1,1,3)$. Let $E=\AA^1$ have weight $k$. A computation using orbifold Reimann-Roch shows that a line bundle on $X\sslash G$ is convex if and only if it is the pullback of a nef bundle on the coarse space of $X\sslash G$. Hence $E_{X\sslash G} = \OO_{\PP(1,1,3)}(k)$ is convex if and only if $k \in 3\ZZ_{\geq 0}$, but $E$ is $I$-convex if and only if $k \geq 0$.
\end{example}

\begin{example}[Convexity does not imply $I$-convexity]\label{ex:weird-one}
Let $X=\AA^5$ and $G = (\Gm)^2$ and let $G$ act on $X$ with weights equal to the columns of the matrix
\[
\left(\begin{array}{ccccc}
1 & 1 & 1 & 3 & 0\\
0 & 0 & 0 & 1 & 1
\end{array}\right).
\]
Let $\theta$ be the product of the two projection characters. Then $X\sslash_\theta G$ is the weighted projective space $\PP(1,1,1,3)$. If we choose $E = \AA^1$ to have weight $(6, 1)$, then $E_{X\sslash G} = \OO_{\PP(1,1,1,3)}(6)$ is convex (here we are using that $\PP(1,1,1,3)$ is also the GIT quotient of $\AA^4$ by $\Gm$ acting with weights $(1,1,1,3)$). 

Let $e_1, e_2$ be the projection characters for $G$ and identify $\tilde \beta \in \Hom(\chi(G), \QQ)$ with $(\tilde \beta(e_1), \tilde \beta(e_2)) \in \QQ^2$. The class $\tilde \beta = (-1/3, 1)$ is $I$-effective but and $(-1/3)6 + 1 = -1$, so $E$ is not $I$-convex.
Hence by Proposition \ref{prop:q-convex}, the nonequivariant limit $\lim_{\kappa \to 0} I^{X\sslash G, E}$ does not exist, despite the fact that $E_{X\sslash G}$ is convex. 
This is initially disturbing; note, however, that no section of $E_{[X/G]}$ has smooth vanishing locus on $X\sslash G$, so the data $(X, G, \theta, E)$ is not part of a CI-GIT presentation.

It is not clear to us if, for a CI-GIT presentation, convexity implies $I$-convexity.

\end{example}

\section{A counterexample to Conjecture \ref{conj:ql}}

\label{sec:failure}
 In this section we study a general quartic hypersurace in $\PP(1,1,1,3)$. 
A CI-GIT presentation for this target (corresponding to the extended stacky fan considered in \cite[Sec~5.4]{OP}) consists of the following data: 
\begin{itemize}
\item $X$, $T:=G$, and $\theta$ are as in Example \ref{ex:weird-one}
\item $E$ is the 1-dimensional $T$-representation with weight $(4, 1)$
\item $s = x_5( x_2^4 + x_3^4) + x_1x_4$ 
\end{itemize}

Define $Y$ to be the zero locus of $s$, let
$e_1, e_2$ be the projection characters of $T$ and for $ \beta \in \Hom(\chi(T), \QQ)$ set $ \beta_i =  \beta(e_i)\in \QQ$; i.e., we have fixed an isomorphism $\Hom(\chi(T), \QQ) \simeq \QQ^2$. As computed in \cite[Sec~5.4]{OP} the $I$-effective classes are
\begin{equation}\label{eq:toric-ieff}
\{  \beta \in (1/3)\ZZ \times \ZZ \mid 3 \beta_1+\beta_2 \geq 0 \;\text{and}\; \beta_2 \geq 0\} \subset \QQ^2.
\end{equation}
For every $I$-effective $ \beta$, the subspace $X^{ \beta}$ contains the subspace of $X$ where $x_1=x_2=x_3=0$. This locus is in the base locus of the linear system defining $Y$, and hence $X^{ \beta} \cap Y^{ss}(G) \neq \emptyset$.

Let $I_\beta(r(q), z)$ denote the coefficient of $q^\beta$ in $I(r(q), z)$ (see Remark \ref{rmk:nov-rings}).
From \eqref{eq:toric-ieff} we see that the class $(-1,3)$ is effective, but it is not $I$-nonnegative since its pairing with the weight $(4, 1)$ is equal to -1. 
To compute $I_{(-1, 3)}(z)$ using Theorem \ref{thm:Ifunc-formula}, we note that if $\beta = (-1, 3)$ then $X^{ \beta}$ is equal to the subspace of $X$ where $x_1=x_2=x_3=0$. Hence $F_{(-1,3)}(X\sslash T)$ is equal to the substack of the untwisted sector of $X\sslash T$ with $\bmu_3$ isotropy. Since this is contained in the base locus of the linear system defining $Y$, we have $F_{(-1,3)}(Y\sslash T) = F_{(-1,3)}(X\sslash T)$, and the hypothesis in Theorem \ref{thm:Ifunc-formula} is satisfied. We obtain 
\begin{equation}\label{eq:toric-evidence}
I_{(-1,3)}(z) = \frac{1}{3!z^3}[B\bmu_3]_1
\end{equation}
where $1 \in G$ is the identity and $[B\bmu_3]_1$ is the fundamental class of the orbifold point in the untwisted sector $\In{Y\sslash T}_1$.

On the other hand, as shown in \cite[Sec~5.4]{OP} the nonequivariant limit of $\bi^*I(r(q), z)$ exists. If Conjecture \ref{conj:ql} were true we would have that 
\begin{equation}\label{eq:twisted-term}
I_{(-1, 3)}(z) = \lim_{\kappa \to 0} \bi^*I^{X\sslash T, E}_{(-1, 3)}(z) = \lim_{\kappa \to 0} \bi^* \frac{H^3}{3!(H+\kappa)z^3} = \lim_{\kappa \to 0} \frac{0}{3!\bi^*(H+\kappa)z^3}=0.
\end{equation}
Here, $H$ is the hyperplane class on $\PP(1,1,1,3)$ and $\bi^*H^3=0$ because $Y$ has dimension 2.

\begin{lemma}\label{lem:conj-is-false} Conjecture \ref{conj:ql} is false, and the ``$I$-function'' computed for $Y\sslash_\theta G$ in \cite[Sec~5.4]{OP} does not lie on the Lagrangian cone of $Y\sslash_\theta G$.
\end{lemma}
\begin{proof} 
If $P$ is a formal power series in $z$ and $z^{-1}$, let $[P]_+$ denote the truncation of $P$ to nonnegative powers of $z$. For ease of notation set $K(r(q), z):= \lim_{\kappa \to 0}\bi^* I^{X\sslash T, E}(r(q), z)$. We will show that 
\begin{equation}\label{eq:mirror-maps}
[zK(r(q), z) - z]_+ = [zI(r(q), z)-z]_+.
\end{equation}
Since $I$ is on the Lagrangian cone of $Y\sslash_\theta T$, this implies that if $K$ is also on the cone we must have $K=I$. Since the coefficients of $(-1,3)$ differ by the previous discussion, we will have a contradiction.

To show \eqref{eq:mirror-maps}, note that the definitions of $K_{\beta}$ and $I_{\beta}$ differ only when $\beta \in \Hom(\Pic^T(X), \QQ)$ is not $I$-nonnegative. So it is enough to show that only $I$-nonnegative classes contribute to the series in \eqref{eq:mirror-maps}.
Now we use homogeneity of $K$ and $I$ (Remark \ref{rmk:homogeneity}). The character of 
\[\det(T_{X\sslash T})\det(E_{X\sslash T})^{-1}\] 
is $(2, 1)$, and the set \eqref{eq:toric-ieff} is additively generated over $\ZZ_{\geq 0}$ by $(-1/3, 1)$ and $(1,0)$, so we have $\deg(q^{\delta}) = \deg(q^\delta) \geq 1/3$ for an $I$-effective class $\delta \in \Hom(\Pic^T(Y), \QQ)$. Now homogeneity implies $I(r(q), z) = \one+  O(z^{-1})$ and similarly for $K(r(q), z).$ Hence a summand of $I$ or $K$ contributing to the series in \eqref{eq:mirror-maps} will contain $z$ to the power $-1$.

To compute the coefficient of $z^{-1}$ we use homogeneity again: a term of $I$ or $K$ with $z^{-1}$ must also contain a variable $q^{\beta}$ and a class $\alpha \in H^k(\In{Y\sslash T}_{g_{ \beta}^{-1}}; \QQ)$ with $\deg(q^{ \beta}) + k + \age_{Y\sslash T}(g_{ \beta}^{-1})=1$.
The ages of elements of $T$ that have nonempty fixed locus on $Y^{ss}_\theta(T)$ are
\[
\age_{Y\sslash T}(1,1) = 0 \quad \quad \quad \age_{Y\sslash T}(\zeta_3, 1) = 2/3 \quad \quad \quad \age_{Y\sslash T}(\zeta_3^2, 1) = 4/3.
\]
Hence the only possibilities are $\deg(q^{ \beta}) = 1$, $g_{\beta}^{-1} = (1,1)$, and $k=0$; or $\deg(q^{\beta}) = 1/3$, $g_{ \beta}^{-1} = (\zeta_3,1)$, and $k=0$. The only element of \eqref{eq:toric-ieff} with $\deg(q^{ \beta}) = 1/3$ is $(-1/3, 1)$, which is $I$-nonnegative. The elements of \eqref{eq:toric-ieff} with $\deg(q^{\beta}) = 1$ are $(0,1)$ and $(-1,3)$. Of these, only $(-1,3)$ is not $I$-nonnegative, and \eqref{eq:toric-evidence} shows that it does not contribute to $[zI-z]_{+}$ because the codimension $k$ is too high. Likewise \eqref{eq:twisted-term} shows that $[zK_{(-1, 3)} - z]_+=0$.

\end{proof}

\section{Quantum period computations in the literature}\label{sec:criterion}
The paper \cite{OP} uses Conjecture \ref{conj:ql} to compute quantum periods of several Fano orbifolds. These computations are likely valid despite the failure of the conjecture, and in this section we explain exactly what must be checked. We first observe that \cite{SW22a} verifies the only nonabelian calculation in \cite{OP}, so in this section we will only consider abelian setups. We continue to assume that $Y\sslash_\theta T$ is a stack with a CI-GIT presentation $(X, G, \theta, E, s).$

\begin{definition}
Let $(X, T, \theta, E, s)$ be a CI-GIT presentation for $Y\sslash_\theta T$. The stack $Y\sslash_\theta T$ is a \emph{Fano orbifold} if 
\begin{itemize}
\item The stacky locus in $Y\sslash_\theta T$ has codimension at least 2, and
\item The coarse moduli space $\overline{Y}$ of $Y\sslash_\theta T$ is a projective variety such that the anticanonical divisor $-K_{\overline{Y}}$ is an ample $\QQ$-Cartier divisor.
\end{itemize}
\end{definition}

The paper \cite{many} defines the \emph{quantum period} of a Fano orbifold to be a certain series determined by the \emph{J-function} of the orbifold. For us, the following partial definitions of these series will suffice.

\begin{definition}\label{def:qperiod}
The \emph{J-function} of $Y\sslash_\theta T$ is a series $J^{Y\sslash T}(\tau, q, z)$ satisfying 
\[
J^{Y\sslash T}(\mu(q, z), q, z) = I(q, z) \quad \quad \quad\text{where}\quad \quad \quad \mu(q, z) = [zI(q, z)-z]_+.
\]
If  $Y\sslash_\theta T$ is a Fano orbifold, its \emph{quantum period} is a series of the form
\[
[J^{Y\sslash T}(\gamma(t, x), t^{-K^{Y\sslash T}}, 1)]_{\one},
\]
where $\gamma$ is a certain series in formal variables $t$ and $x$ valued in $H^\bullet_{\CR}(Y\sslash T, \QQ)$, the line bundle $K^{Y\sslash T}$ is the canonical bundle of $Y\sslash T$, and $[\cdot]_\one$ means the coefficient of $\one$ in the cohomology-valued series $\cdot$.
\end{definition}

\begin{lemma}\label{lem:op-ok}
Let $(X, T, \theta, E, s)$ be a CI-GIT presentation for a Fano orbifold with associated $I$-function $I(r(q), z)$. Assume that if $ \delta \in \Hom(\Pic^T(Y), \QQ)$ is $I$-effective and nonzero, then $\deg(q^{ \delta}) > 0$. Assume moreover that for all $ \beta \in \Hom(\chi(T), \QQ)$ that are not $I$-nonnegative,
\begin{enumerate}
\item The limit $\lim_{\kappa \to 0} \bi^* I^{X\sslash T, E}_{\beta}$ exists and equals zero.
\item The degree of the Poincar\'e dual of $[F_{ \beta}(Y\sslash T)]$ is at least $1$; that is,
\[
\age(g_{ \beta}^{-1}) + \mathrm{codim}_{\In{Y\sslash T}_{g_{ \beta}^{-1}}} F_{ \beta}(Y\sslash T) \geq 1
\] 
\end{enumerate}
Then the ``quantum period'' computed from $I^{X\sslash T, E}$ using Conjecture \ref{conj:ql} is equal to the quantum period for $Y\sslash T$.
\end{lemma}
\begin{example}\label{ex:this-works}
Lemma \ref{lem:op-ok} shows that the quantum period computation in \cite[Sec~5.4]{OP} for the CI-GIT target considered in Section \ref{sec:failure} is correct, despite the discrepancy in $I$-functions.
\end{example}
\begin{remark}
A sharper version of Lemma \ref{lem:op-ok} also holds: in condition (2), one can replace $\mathrm{codim}_{\In{Y\sslash T}_{g_{ \beta}^{-1}}} [F_{ \beta}(Y\sslash T)]$ with the codimension of $[F_{ \beta}(Y\sslash T)]^{\vir}$ in $\In{Y\sslash T}_{g_{ \beta}^{-1}}$.
\end{remark}

\begin{proof}[Proof of Lemma \ref{lem:op-ok}]
The assumption on $\deg(q^\delta)$, together with homogeneity of $I$ and $I^{X\sslash T, E}$ (Remark \ref{rmk:homogeneity}), give us that $I$ and $\lim_{\kappa \to 0} \bi^*I^{X\sslash T, E}$ both have the form $\one + O(z^{-1})$. It follows that the \textit{mirror maps}
\[
\mu(r(q), z) := [zI(r(q), z)  - z]_+ \quad \quad \quad \mu^{\tw}(r(q),z) = [z\lim_{\kappa \to 0} \bi^* I^{X\sslash T, E}(r(q), z) - z]_+
\]
are determined by the coefficients of $z^{-1}$ in $I$ and $\lim_{\kappa \to 0} \bi^*I^{X\sslash T, E}$, respectively. The significance of $\mu^{tw}$ is that 
if Conjecture \ref{conj:ql} were true, we would have a formula
\begin{equation}\label{eq:bad-mirror}
J^{Y\sslash T}(\mu^{\tw}(r(q), -z), r\circ \iota(Q), z) = \lim_{\kappa \to 0} \bi^* I^{X\sslash T, E}(r(q), z)
\end{equation}
where $J^{Y\sslash T}(\tau, Q, z)$ is the $J$-function of $Y\sslash T$ (defined as in \cite{SW22a}). 
This is the formula used (implicitly) in \cite{OP} to compute the quantum period. The mirror theorem of \cite{yang} says that \eqref{eq:bad-mirror} does hold when we replace $\mu^{\tw}$ with $\mu$ and $\lim_{\kappa \to 0} \bi^*I^{X\sslash T, E}$ with $I$ (see also Definition \ref{def:qperiod}).

We first show that $\mu=\mu^{tw}$; i.e., we check that the coefficients of $z^{-1}$ in $I_\beta$ and $\lim_{\kappa \to 0} \bi^*I_\beta^{X\sslash T, E}$ agree for every $I$-effective $\beta \in \Hom(\chi(T), \QQ)$. If $\beta$ is $I$-nonnegative then it follows from the definitions that $\lim_{\kappa \to 0} \bi^*I_\beta^{X\sslash T, E}$ exists and equals $I_\beta$.
Any terms with $q^{ \beta} z^{-1}$ in $\lim_{\kappa \to 0} \bi^*I^{X\sslash T, E}$ where $ \beta$ is not $I$-nonnegative are zero by (1); we show that the same holds for $I$. By homogeneity, a nonzero term of the form $\alpha q^{ \beta} z^{-1}$ in $I(r(q), z)$ must have $\alpha \in H^k(\In{Y\sslash G}_{g_{ \beta}^{-1}}; \QQ)$ where $\deg(q^{ \beta}) + k + \age_{Y\sslash T}(g_{ \beta}^{-1})=1$. By the restriction on $\deg(q^{\beta})$ we have $k+\age_{Y\sslash T}(g_{ \beta}^{-1}) < 1$. This contradicts condition (2) of the lemma, since by Theorem \ref{thm:Ifunc-formula} the class $\alpha$ is supported on $F_{ \beta}(Y\sslash T) \subset \In{Y\sslash T}_{g_{ \beta}^{-1}}$.

Since $\mu = \mu^{tw}$, if Conjecture \ref{conj:ql} were true we would have that
\[
J^{Y\sslash T} (\mu(q, -z), r\circ \iota(Q), z) = \lim_{\kappa \to 0} \bi^* I^{X\sslash T, E}(r(q), z) = I(r(q), z).
\]
In reality the series $\lim_{\kappa \to 0} \bi^* I^{X\sslash T, E}(r(q), z)$ and $I(r(q), z)$ may differ (see Section \ref{sec:failure}). However the quantum period only depends on the coefficient of $\one$ in these series. Once again, these coefficients agree for $I$-nonnegative $\beta$, and if $ \beta$ is not $I$-nonnegative then $\lim_{\kappa \to 0} \bi^*I^{X\sslash T, E}=0$, so it is enough to check that if $ \beta$ is not $I$-nonnegative and $g_{ \beta = 1}$ then $I_{ \beta}=0$. This follows from condition (2) since $\age(\one)=0$.

\end{proof}

\printbibliography
\end{document}